\documentclass[11pt]{amsart}
\usepackage{amscd,skak} 
\usepackage[all,cmtip]{xy}               % See geometry.pdf to learn the layout options. There are lots.
\numberwithin{equation}{section}
\usepackage{xcolor}

\usepackage{cancel}

\usepackage{graphicx}
\usepackage{amssymb}
\usepackage{MnSymbol}

\DeclareMathAlphabet{\mathpzc}{OT1}{pzc}{m}{it}

\newtheorem{proposition}{Proposition}[section]
\newtheorem{definition}[proposition]{Definition}
\newtheorem{lemma}[proposition]{Lemma}
\newtheorem{theorem}[proposition]{Theorem}

\newcommand{\I}{\left(0,\frac \pi 2\right)}

\DeclareMathOperator{\V}{\mathcal V^\infty}

\DeclareMathOperator{\Tan}{\operatorname {Tan}}

\DeclareMathOperator{\Nor}{\operatorname{Nor}}

\DeclareMathOperator{\supp}{supp}
\DeclareMathOperator{\z}{(zero  \ section)}
\newcommand{\restrict}[2]{\left. #1 \right|_{#2}}
\newcommand{\R}{\mathbb{R}}
\newcommand{\Rn}{\mathbb{R}^n}
\newcommand{\Rk}{\mathbb{R}^k}
\newcommand{\Rl}{\mathbb{R}^l}
\newcommand{\Rm}{\mathbb{R}^m}

\newcommand{\pol}{\mathcal{P}}
\newcommand{\m}{\mathcal{M}}

\newcommand{\eps}{\epsilon}

\newcommand\proj{\operatorname{proj}}

\newcommand{\ddt}{{\frac d{dt}}}
\newcommand{\ddtzero}{\restrict{\frac d{dt}}{t=0}}

\newcommand \inv{^{-1}}
\newcommand\piinv{\pi^{-1}}

\newcommand \lcur{[\![} 
\newcommand\rcur{]\!]}
\newcommand \cur[1]{\lcur #1\rcur} 
\newcommand \meas[1]{[ #1]}

\newcommand\C{\mathcal C}

\title{Intersection theory and the Alesker product}
\author{Joseph H.G. Fu}
 \address{ Department of Mathematics, 
University of Georgia, 
Athens, GA 30602, USA}
\email{fu@math.uga.edu}

\date{\today}                                          
\thanks{ Partially supported by NSF grant DMS-1406252.}

\begin{document}
\begin{abstract}
Alesker has introduced the space $\V(M)$ of  {\it smooth valuations} on a smooth manifold $M$, and shown that it admits a natural commutative multiplication. Although Alesker's original  construction is highly technical, from a moral perspective this product is simply an artifact of the operation of intersection of two sets.
Subsequently Alesker and Bernig gave an expression for the product in terms of differential forms. We show how the Alesker-Bernig formula arises naturally from the intersection interpretation, and apply this insight to give a new formula for the product of a general valuation with a valuation that is expressed in terms of intersections with a sufficiently rich family of smooth polyhedra.

%, with the consequence that the space $\cc(M)$ of smooth curvature measures on $M$ is naturally a module over $\V(M)$. 
%Associated to any Riemannian structure on $M$ there is a subalgebra $\LK(M) \subset \V(M)$, isomorphic to a truncated polynomial algebra, spanned by the valuations arising from the Weyl tube formula. Meanwhile, certain classical curvature integrals for Riemannian manifolds give rise to a subspace $\rr(M) \subset \cc(M)$. We show that $\rr(M)$ is naturally a module over $\LK(M)$. In fact this module is the image of a certain universal module over the polynomial algebra $\R[t]$. We determine the structure of this universal module and give some applications.

\end{abstract}

\maketitle
%\tableofcontents

\setcounter{tocdepth}{200}

\section{Introduction} 
S. Alesker has introduced the space $\V(M)$ of {\it smooth  valuations}  on a smooth manifold $M$, whose elements are finitely additive set functions defined for sufficiently regular compact subsets of a smooth manifold.  Furthermore $\V(M)$ carries a natural multiplicative structure, for which the Euler characteristic $\chi$ acts as the identity element.  These ideas provide a language that has transformed modern integral geometry (cf. \cite{bernig1, conv, hig, bfs, futocome}).

%\subsection{The predicament} 
The basic idea behind the Alesker product is very simple. Given a sufficiently regular subset $X\subset M$, we may define the functional $\nu_X$ by
$$
\nu_X(A):= \chi(\,A \cap X).
$$
Although this can only be defined for subsets $A$ that meet $X$ in a nice way, it is clear that $\nu_X$ is indeed finitely additive under this restriction. Thus $\nu_X$ is not a smooth valuation, but only a {\it generalized valuation} in the sense of \cite{ale-be09}. On the other hand, a basic principle  states that a smooth valuation may be approximated by linear combinations of valuations of this form (if $M$ is a finite dimensional real vector space and attention is restricted to the translation-invariant elements of $\V(M)$ then this is a famous conjecture of P. McMullen, proved by Alesker in \cite{ale01}).
If we now define
\begin{equation}\label{eq:heuristic}
\nu_X \cdot \nu_Y := \nu_{X\cap Y},
\end{equation}
or more generally
\begin{equation}\label{eq:heuristic2}
(\nu_X \cdot \mu)(A): = \mu(\,A \cap X),
\end{equation}
then the idea behind the procedure of \cite{ale01,ale04,alefu} is to show that this product extends by linearity and continuity to all of $\V(M)$. 

Following Alesker we use the term {\it smooth polyhedron} to refer to a properly embedded smooth submanifold of $M$ with corners. By definition a {\it smooth valuation} is given in terms of a pair $\gamma,\beta$ of differential forms living respectively on $M$ and on its cosphere bundle $S^*M$, and assigns to any compact smooth polyhedron $A \subset M$ the sum of an interior term $\int_A \gamma$ and a boundary term $\int_{N(A)}\beta$. Here $N(A) \subset S^*M$ is the {\it conormal cycle } of $A$,  the manifold of  local supporting tangent hyperplanes for $A$. Alesker and Bernig \cite{ale-be09}  distilled the definition of the Alesker product of two smooth valuations into a formula involving only the differential forms underlying the factors and deduced the relation \eqref{eq:heuristic}.

However, the technicalities thus introduced, both in Alesker's original approach and in the approach of Alesker-Bernig, are significant, to the point where the basic simplicity of the construction of the product is obscured. For example, in the study \cite{bfs} of kinematic formulas in complex space forms a rather obscure argument was needed to prove the following essentially simple fact. Recall that $(M,G)$ is called a {\it Riemannian isotropic space} if $M$ is Riemannian and $G$ is a Lie group of isometries of $M$ that acts transitively on the tangent sphere bundle. Denote by $dg $ the Haar measure of $G$.
\begin{theorem}[\cite{bfs}, Theorem A.1]\label{thm:bfs}
Let $(M,G)$ be an isotropic space,  $X$  a compact smooth polyhedron, and $\rho \in C^\infty(G)$. Then
\begin{equation}\label{eq:g nu def}
\nu(A):= \int_G \chi(A \cap gX) \, \rho(g) \, dg 
\end{equation}
defines a smooth valuation on $M$.
If $\phi \in \V(M)$ then the Alesker product of $\phi$ and $\nu$ is given by
\begin{equation}\label{eq:G asymmetric product}
(\phi\cdot \nu) (A)= \int_G \phi (A \cap gX) \, \rho(g)\, dg.
\end{equation}
\end{theorem}

\subsection*{Results of this paper} 
The main result of the present paper (actually an immediate consequence of Theorems \ref{thm:well defined} and \ref{thm:main}) generalizes Theorem \ref{thm:bfs}, as follows. We replace $M$ by a general smooth oriented manifold, $(G, dg)$ by a smooth oriented parameter manifold $P$ equipped with a smooth signed volume form $dp$, and the action of $G$ on $M$ by a smooth family of orientation-preserving diffeomorphisms $\{\varphi_p\}_{p \in P}$ of $M $. We assume that this apparatus is proper in a sense generalizing the usual notion of a proper group action, and that the induced diffeomorphisms $\tilde \varphi_p$ of the cosphere bundle $S^*M$ satisfy the condition that for any $\xi \in S^*M$ the map $P \to S^*M$ given by $p \mapsto \tilde \varphi_p(\xi)$ is a submersion. Under these conditions we say that $(P,\varphi,dp)$ is an {\it admissible} measured family of diffeopmorphisms of $M$, generalizing the Riemannian isotropic condition (cf. Definition \ref{def:kin val} \eqref{item:admissible} below). 
\begin{theorem}\label{thm:featured} Let $(P,\varphi,dp)$ be an admissible measured family of diffeomorphisms of $M$, and $X \subset M$ a compact smooth polyhedron. Then the set function
\begin{equation}\label{eq:kin val 0}
\nu: A \mapsto \int_P \chi(A \cap \varphi_p(X)) \, dp
\end{equation}
defines a smooth valuation on $M$. Given another valuation $\mu \in \V(M)$, we have
\begin{equation}
(\mu \cdot \nu)(A):= \int_P \mu(A \cap \varphi_p(X)) \, dp.
\end{equation}
\end{theorem}

A smooth valuation of the form \eqref{eq:kin val 0} will be called a {\bf principal kinematic valuation}.

The proof is accomplished by assigning a geometric meaning to each of the four terms that arise in the Alesker-Bernig formula in case one of the factors has the form \eqref{eq:kin val 0}. One is a straightforward interior term. The other three arise from a 
natural decomposition into three pieces of the normal cycle $N(A\cap X)$ of the  intersection of two smooth polyhedra $A,X$ in general position: the piece of $N(A)$ lying above the interior of $X$, the piece of $N(X)$ lying above the interior of $A$, and a piece that lies above the intersection of the boundaries of $A,X$  obtained by interpolating the family of arcs between the respective outward normals to $A,X$. This is accomplished in Theorems \ref{thm:well defined} and \ref{thm:nu chi}.
This decomposition is central to classical integral geometry, e.g. in the classical proof of the kinematic formula  given in  \cite{santalo}, III.15.4.

\subsection*{Concluding introductory remarks} A major component of our motive here is foundational and pedagogical. By analogy with the McMullen conjecture, it is conceivable that every smooth valuation is a principal kinematic valuation, or at any rate that the principal kinematic valuations are  dense in $\V(M)$ in a sense strong enough to imply that the product of two smooth valuations may be regarded as the limit of the products of sequences of approximating kinematic valuations (this is true of invariant valuations in an isotropic space; cf. \cite{bfs}, Corollary 2.18). If and when these statements are established it will be possible to view the Alesker-Bernig formula, which may even now be taken as the definition of the Alesker product, as a direct consequence of our intersection formulas. This would offer a number of expository advantages: for instance, the commutativity of the product would then be a direct consequence of Fubini's theorem.

As a final remark, a significant issue in the theory of valuations is the question of exactly how regular a compact subset $A \subset M$ must be in order to possess a conormal cycle and thereby to be amenable to pairing with a smooth valuation. For example, it is known that semiconvex or subanalytic sets enjoy this regularity, but it is also clear that a much wider and unfathomed range is possible (cf. e.g. \cite{pok-rat}). By the same token the subset $X$ used in the construction of the kinematic valuation $\nu$ may in principle be selected from this range. At the cost of introducing more sophisticated technology, all of the main results in the present paper extend to the framework of \cite{fu94}, which includes the cases where $X,A$ are semiconvex or even WDC in the sense of \cite{pok-rat} and \cite{FPR}; we intend to explain this point more fully elsewhere. The point of the present paper is to work entirely within the more familiar framework of the $C^\infty$ category, for the sake of greater accessibilty. Thus our hope is that the discussion will be comprehensible to any reader with a thorough grounding in the basic constructions of differential geometry. Because of this we have omitted, or only sketched, a number of proofs that appear to us to be straightforward from that perspective.

\subsection{Acknowledgements} It is a pleasure to thank the Institut f\"ur Diskrete Mathematik und Geometrie at the Technische Universit\"at of Vienna for their kind hospitality as the technical outline of this paper was sketched. I would also like to thank T. Wannerer for helpful conversations, as well as A. Bernig and G. Solanes for the collaborative project \cite{bfs} that made the urgency of this project felt. Thanks are due to S. Alesker, A. Bernig, D. Faifman,  F. Schuster, and G. Zhang for their helpful comments on an earlier version of this paper. 

\section{Currents and differential forms} We collect a few well known facts and fix notation and conventions. 
Throughout this paper
 $M$ will denote an oriented smooth manifold of dimension $n$. The orientation is not strictly necessary in order to develop the theory of smooth valuations, but it definitely simplifies the discussion.

 \subsection{Forms, currents, pre-images, intersections, fiber integration} \label{sect:general}Formally, our entire discussion revolves around the duality between currents and differential forms. We devote extra care to determining the signs of intersections.
 
\subsubsection{} Denote by $\Omega^k(M)$ the space of smooth differential forms of degree $k$ on $M$, and by $\Omega^k_c(M)$ the subspace of compactly supported forms. A  {\bf current of dimension $k$} on $M$ is a linear functional $\Omega_c^k(M)\to \R$ that is continuous with respect to $C^\infty$ convergence with uniformly compact support. We denote the pairing of a current $T$ with a differential form $\beta $ by $\langle T,\beta\rangle$; indeed we will adopt the same notation for the pairing of any vector and covector.
A properly embedded smooth oriented submanifold $X$ of dimension $k$ determines a current of dimension $k$ by 
$\beta \mapsto \int_X \beta.$
Following the convention of \cite{derham}, any  $\omega \in \Omega^{n-k}(M)$ determines a  current $T$ of dimension $k$ by
$\langle T,\beta\rangle:= \int_M   \omega\wedge \beta$. A current of this form is called {\bf smooth}, and $\omega$ is the {\bf associated differential form} to $T$.

\subsubsection{}\label{sect:orient intersect} Given smooth oriented submanifolds $X^{n-k},Y^{n-m} \subset M^n$ that intersect transversely, we denote by $X\bullet Y $ their oriented  intersection.  The convention determining the orientation is as follows,  equivalent to that of \cite{guil-pol}, Chapter 3.2. Given $x \in X\cap Y$ let $v_1,\dots,v_{m}, u_1,\dots, u_{n-k-m},w_1,\dots,w_k$  be a basis for $T_xM$ such that the $v_i,u_j \in T_xX$, the $u_j \in T_x(X\cap Y)$ and the  $u_j,w_l\in T_xY$. Then the orientation of $X\bullet Y$ is determined by the condition that the product of the orientations of these four ordered bases is $+1$.  If $m+k=n$ (i.e. the intersection is zero-dimensional) then this same convention defines the multiplicity of the intersection. In this case we put $\#(X\bullet Y)= \int_{X\bullet Y} 1$ for the sum of these multiplicities at the various points of intersection (provided they are finite in number). A straightforward calculation reveals that
\begin{equation}\label{eq:commute intersect}
X\bullet Y= (-1)^{mk} (Y\bullet X)
\end{equation}
in the obvious sense.

The symbol $\cap$ will be reserved for set-theoretic intersection.

\subsubsection{Fiber integration and slicing} \label{subsect:preimages} Suppose $M,M'$ are smooth oriented manifolds of dimensions $n,n+m$ respectively, and $\pi:M'\to M$ a smooth submersion. If $\beta \in \Omega^l(M')$ and $\restrict \pi{\supp \beta}$ is proper then there is a well defined fiber integral 
$$
\pi_*\beta \in \Omega^{l-m}(M)
$$
characterized uniquely by the condition
\begin{equation}\label{eq:def fiber int}
\int_M  \gamma \wedge(\pi_* \beta )= \int_{M'} \pi^*\gamma \wedge \beta
\end{equation}
for $\gamma \in \Omega_c^{n+m-l}(M)$. This agrees with the convention \eqref{eq:def fiber int} of \cite{ale-be09}.

By Ehresmann's fibration theorem, the fibers $F_x:=\pi\inv(x)$ may  be oriented so that $M'$ is locally identified with the  oriented product  $M\times F_x$ in the neighborhood of $F_x \subset M'$. If $X \subset M$ is a smooth oriented submanifold then its preimage $\pi\inv X \subset M'$ is then oriented so as to agree locally with the oriented products  $X\times F_x$. Thus $\pi\inv M = M'$ as oriented manifolds.

This orientation convention also entails the following lemma about ``slicing."
\begin{lemma}\label{lem:slicing} Let $\lambda:M'\to M$ be a smooth proper map, and  $N \subset M$ the subset of critical values of $\lambda$, i.e. the set of points where the derivative of $\lambda$ has rank less than $n$. Then $N$ has measure zero in $M$, and $M'\setminus \lambda\inv(N)$ is open. 

If $\gamma\in \Omega^n(M)$ is a smooth differential form of top degree, and $\beta \in \Omega_c^\ell(M')$, then for any properly embedded oriented  smooth submanifold $Y \subset M'$ of dimension $n+\ell$
\begin{equation}\label{eq:slicing}
\int_{Y} \lambda^* \gamma \wedge \beta = \int_{M\owns p} \left(\int_{Y \bullet\lambda\inv(p)} \beta \right) \, \gamma
\end{equation}
\end{lemma}
\begin{proof} The first assertion is Sard's theorem. The second follows from the properness of $\lambda$ and the fact that the set $C \subset M'$ of critical points is closed.

To prove \eqref{eq:slicing}, we may apply the first paragraph with $M'$ replaced by $Y$ and $\lambda $ replaced by its restriction to $Y$. Put $N'$ for the set of critical values of $\restrict \lambda Y$. Then \eqref{eq:slicing} holds if $M$ is replaced by $M\setminus N'$ and $Y$ by $Y\setminus \lambda\inv(N')$: this follows from the implicit function theorem and the orientation convention above for preimages. Put $C' \subset Y$ for the set of critical points of $\restrict \lambda Y$. Then
\begin{equation*}
\int_{\lambda\inv(N') } \lambda^* \gamma \wedge \beta  = \left(\int_{\lambda\inv(N'\setminus C') } + \int_{C'}\right) ( \lambda^* \gamma \wedge \beta), 
\end{equation*}
where the first integral vanishes by the Sard's theorem and the implicit function theorem and the integrand of the second integral vanishes identically.
\end{proof}

\subsubsection{} Under these conventions one may easily check the following properties by restricting to the case that $M'$ is a product $M\times F$. 
\begin{lemma}\label{lem:preimages}{\ }
\begin{enumerate}
\item \label{item:pushdown intersect} If $X\subset M, Y\subset M'$ are smooth oriented submanifolds,
%such that $\dim X +\dim Y =n$
 and the restriction of $\pi $ to $Y$ is a diffeomorphism such that $X,\pi(Y)$ meet transversely, then
\begin{equation}\label{eq:intersect project}
\pi(Y\bullet\pi\inv X )=\pi(Y) \bullet X.
\end{equation}
\item If $X,Z \subset M$ are oriented submanifolds intersecting transversely then
\begin{equation}\label{eq:pullback intersect}
\pi\inv (X\bullet Z) = (\pi\inv X)\bullet (\pi\inv Z).
\end{equation}
\item
If $\beta \in \Omega_c^{l}(M')$ and $\dim X = l-m$ then 
 \begin{equation} \label{eq:fiber int over z}
 \int_X  \pi_* \beta = \int_{\pi\inv X} \beta.
 \end{equation}
\item The current defined by integration over $X$ may be approximated weakly by a sequence of currents defined by differential forms $\omega_1,\omega_2,\dots$, with supports converging in the Hausdorff metric topology to $X$. The current defined by integration over $\pi\inv X$ is then the weak limit of the sequence of currents defined by the differential forms $\pi^* \omega_1,\pi^*\omega_2,\dots$, with supports converging to $\pi\inv X$.
\item Given another oriented manifold $M''$ and a second submersion $\lambda:M''\to M'$ we have the identity of oriented manifolds
\begin{equation}\label{eq:preimage composition}
\lambda\inv(\pi\inv X) = (\pi\circ\lambda)\inv X.
\end{equation}
%\item Suppose in addition that  $\pi:M'\to M$ is a fiber bundle $\mathcal B$, $L$ is a smooth manifold, and $\zeta:L \to M$ is a smooth map. Put $L'$ for the total space of the pullback bundle $\zeta^*\mathcal B$, with projection $\lambda:L'\to L$ and induced map $\zeta':L'\to M'$. Let $Z \subset L, X \subset M$ be  smooth oriented submanifolds such that $\restrict \zeta Z$ is an orientation-preserving diffeomorphism onto $X$. Then
%$\restrict {\zeta'}{\lambda\inv Z}$ is an orientation-preserving diffeomorphism onto $ \pi\inv X$.
\end{enumerate}
\end{lemma}

The last statement  follows from the approximation property above and the corresponding fact about differential forms.

\subsection{Smoothing a current by a family of diffeomorphisms}  

We now describe a  procedure  for smoothing a current given by integration over a submanifold that will be central to the rest of the paper.
 
\begin{lemma} \label{lem:basics}
Let $P$ be an oriented smooth  manifold, equipped with a smooth signed volume form $dp$, and $\phi:P\times M \to M$ a smooth map such that 
\begin{itemize}
\item each $\phi_p:= \phi(p,\cdot)$  is a diffeomorphism
\item for each $x \in M$, the map $\phi^x:= \phi(\cdot,x)$  is a submersion
\item the restriction to $\supp dp \times M$ of the map $(p,x) \mapsto (\phi_p(x), x)$ is proper.
\end{itemize}
Suppose $\pi:M'\to M$ is a submersion from a second smooth oriented manifold $M'$, with $\dim M'= N$.
% and that there is a smooth map $\phi':P\times M' \to M'$ that is intertwined with $\phi$ via $\pi$. 
Fix  a properly embedded oriented smooth submanifold  $X \subset M$ of codimension $k$. Then:
\begin{enumerate}
\item\label{item:basic2} The linear functional  $\beta \mapsto \int_P\int_{ \phi_{p}( X)}\beta \, dp $ is a well defined smooth current  on $M$ of dimension $n-k$. Put $\omega\in \Omega^{k}(M)$ for the associated differential form.
\item\label{item:basic1} Let $Y\subset M'$ be a compact smooth oriented submanifold of codimension $m$. Then  $Y$  intersects $\pi\inv\phi_p(X)$  transversely for a.e. $p\in P$. 
If  $\beta \in \Omega^{N-m-k}(M)$, then
\begin{equation}\label{eq:intersect integrate 3}
\int_{Y}\pi^* \omega \wedge \beta  = \int_P \left(\int_{Y\bullet(\pi\inv\phi_p(X) ) } \beta \right)\, dp.
\end{equation}
In particular, taking $M'=M$ and $\pi$ to be the identity map,
\begin{equation}\label{eq:intersect integrate}
\int_Y \omega = \int_P \#( Y\bullet\phi_p(X) ) \, dp
\end{equation}
if $\dim Y =N-m= k$. More generally,
\begin{equation}\label{eq:intersect integrate 2}
\int_Y \omega \wedge\beta  = \int_P \left(\int_{Y\bullet \phi_p(X)  } \beta \right)\, dp.
\end{equation}

%\item\label{item:basic3} 
%% Then for $Y'\subset M'$ compact, smooth and oriented, the intersection of $\pi\inv \phi_p(X)$ and $Y'$ is transverse for almost every $p \in P$, with
%%$$
%%\int_{Y'} \pi^*\omega = \int_P (\pi\inv\phi_p(X)) \bullet Y' \, dp
%%$$
%%where the integral on the right is absolutely convergent. More generally, 
%If $Z'\subset M'$ is a smooth oriented manifold of dimension $ m \ge  k$, and $\beta \in \Omega^{m-k}(M)$, then
\end{enumerate}
\end{lemma}
\begin{proof}

\eqref{item:basic2}: This is a special case of the following more general fact: if $T$ is a current of dimension $n-k$ on $M$ then $\bar T:= \int_P \phi_{p*} T \, dp$, given by
$$
\langle \bar T,\beta\rangle := \int_P \langle  T ,\phi_{p}^*\beta\rangle \, dp,
$$
is smooth. To prove this, note first that using a partition of unity we may assume that $M= \Rn$, and recall the fundamental fact that if $K\in C^\infty(\Rn \times \Rn)$, such that the restriction to $\supp K$ of the projection to the first factor is proper, then for any distribution $h$ on $\Rn$ the function
$$
x \mapsto  \langle h (y), K(x,y)\rangle 
$$
is smooth. Thinking of $T$ as a differential form with distributional coefficients $h$, \eqref{item:basic2} follows directly, since the coefficients of $\bar T $ may be expressed as a sum of functions of this type corresponding to kernels $K$ obtained via the coarea formula from $dp$ and the derivatives of $\phi$.

\eqref{item:basic1}: The first assertion follows at once from Theorem 6.35 of \cite{lee}.

Let $\psi_1,\psi_2,\dots \in \Omega^{m}(M')$ be smooth differential forms whose associated currents converge weakly to $Y$. The left hand side of \eqref{eq:intersect integrate 3} is then the limit of the expressions

\begin{align*}
\int_{M'} \psi_i \wedge \pi^*\omega \wedge \beta &= (-1)^{mk} \int_{M'} \pi^*\omega \wedge \psi_i\wedge \beta\\
&=  (-1)^{mk} \int_{M}\omega \wedge \pi_*( \psi_i\wedge \beta)\\
&= (-1)^{mk} \int_P \int_{\phi_p(X)} \pi_*(\psi_i \wedge \beta) \, dp \\
&= (-1)^{mk} \int_P \int_{\pi\inv\phi_p(X)} \psi_i \wedge \beta \, dp
\end{align*}
The submersivity condition on $\phi$ implies that
\begin{equation*}
L:= \{(x,y,p) \in X \times M'\times P: \pi(y) = \phi_p(x)\}.
\end{equation*}
is a smooth manifold, and that the projections $\mu:L\to M',  \lambda:L \to P$ are smooth submersions. Clearly each $\lambda\inv(p )\simeq \pi\inv \phi_p(X)$; we orient $L$ so that the orientations agree with that induced by that of $X$. On the other hand, by Lemma \ref{lem:slicing} and \eqref{eq:commute intersect},
\begin{align*}
\int_P \int_{\pi\inv\phi_p(X)} \psi_i \wedge \beta \, dp&= \int_P \int_{\lambda \inv(p)} \mu^*(\psi_i \wedge \beta) \, dp\ \\
&= \int_L  \lambda^* dp\wedge \mu^*\psi_i \wedge \mu^* \beta   \\
&=(-1)^{m\dim P} \int_L \mu^*\psi_i\wedge  \lambda^* dp \wedge \mu^* \beta   \\
&\to (-1)^{m\dim P} \int_{\mu\inv Y}  \lambda^* dp \wedge \mu^* \beta   \\
&= (-1)^{m\dim P}\int_P\int_{\mu\inv Y \bullet \lambda\inv(p)} \mu^*\beta \\
&= \int_P\int_{ \lambda\inv(p) \bullet \mu\inv Y} \mu^*\beta \\
&= \int_P\int_{ \pi\inv\phi_p(X) \bullet  Y} \beta\\
&= (-1)^{mk}\int_P\int_{  Y\bullet  \pi\inv\phi_p(X)} \beta
\end{align*}
as $i \to \infty$.
\end{proof}

\section{Smooth polyhedra and conormal cycles} 
\subsection{The cotangent and cosphere bundles}\label{sect:cotan} 
We denote by $S^*M$ the cosphere bundle of $M$, which we may regard either as the space of oriented tangent hyperplanes to $M$ or else as the quotient of the deleted cotangent bundle $T^*M-\z$ under the equivalence relation $\bar \xi \sim t\bar \xi$ for $t >0$. Given $0\ne\bar \xi \in T^*M$ we denote its image in $S^*M$ by $[\bar \xi]$ or $\proj(\bar \xi)$. The projections of $T^*M$ and $S^*M$ to $M$ will be denoted by $\pi$. For convenience (only) we will sometimes impose an arbitrarily chosen Riemannian metric on $M$, in which case $S^*M$ may be identified with the tangent sphere or cotangent sphere bundle of $M$.

\subsubsection {}\label{sect:norient}The cotangent bundle $T^*M$ is canonically oriented. On the other hand there are several distinct natural ways to orient $S^*M$. We will fix the  orientation that agrees with the orientations coming from the local product structure $\Rn \times S^{n-1}$. Note that this  is $(-1)^n$ times the orientation as the 
 boundary of the submanifold of $T^*M$ consisting of all covectors of  length $\le 1$ with respect to some Riemannian metric (recall that this orientation is determined by the condition that the sign of an ordered basis $v_1,\dots,v_{2n-1}$ for $T_{\xi}S^*M$ agrees with that of the ordered basis $\frac{\partial}{\partial r},v_1,\dots,v_{2n-1}$ for $T_\xi T^*M$, where $\frac{\partial}{\partial r}$ is the Euler vector field). The fibers of the projectivization map $\proj:T^*M-\z \to S^*M$ may be identified with $(0,\infty)$ via the associated length function. Under the orientation convention of Section \ref{sect:general}, the orientations of these fibers are $(-1)^{n-1}$ times the canonical orientation of $(0,\infty)$.

\subsubsection{} A diffeomorphism $\varphi:M \to M$ induces a canonical symplectomorphism $(\varphi\inv)^*:T^*M \to T^*M$, homogeneous on fibers, via pullback, and thereby a contactomorphism $\tilde \varphi:S^*M \to S^*M$ that intertwines $\varphi $ and the projection $\pi:S^*M \to M$.

\subsubsection{} If $\xi,\eta \in S_xM$ and $\xi \ne -\eta$ then there is a well-defined {\bf segment} $\overline{\xi,\eta}\subset S_xM $ joining $\xi $ to $\eta$, defined as the set of all $[\cos t \,\bar \xi + \sin t\, \bar \eta], t \in (0,\frac \pi 2)$, for some arbitrarily chosen representatives $\bar \xi,\bar \eta$ with $[\bar \xi]= \xi, [\bar\eta]= \eta$. Under the identifications induced by a Riemannian metric as above, we may think of $\overline{\xi,\eta}$ as the minimizing geodesic from $\xi$ to $\eta$ in the (co)tangent sphere.

\subsubsection{} For convenience we choose a contact form $\alpha \in \Omega^1(S^*M)$: selecting a smooth global section $\sigma$ of the ray bundle $T^*M-\z \to S^*M$ we put for $\xi \in S^*M$
$$
\langle\alpha_{\xi}, \tau\rangle = \langle \sigma(\xi),\pi_* \tau\rangle.
$$
The choice of $\alpha$ determines a unique {\bf Reeb vector field} $T$ on $S^*M$ via the conditions
$$
\langle \alpha, T\rangle = 1, \quad \mathcal L_T\alpha = 0,\quad T \without \,d\alpha = 0.
$$
where $\mathcal L$ denotes the Lie derivative.
If $\beta \in \Omega^*(S^*M)$ then
\begin{equation}\label{eq:reeb}
\beta = \alpha \wedge(T \without \, \beta ) + T \without \, (\alpha \wedge \beta).
\end{equation}

It follows that $\beta$ is a multiple of $\alpha$ iff $\alpha \wedge \beta = 0$. Such a form $\beta$ is sometimes said to be {\it vertical} with respect to the contact structure. Observe that if $\omega \in \Omega^n(M)$ then $\pi^* \omega$ is vertical in this sense.

\subsection{Smooth polyhedra and conormal cycles}\label{sect:normal cycles}
\begin{definition} A {\bf smooth polyhedron} in $M$ is a properly embedded smooth submanifold with corners, i.e. a closed subset $A\subset M$ such that  each $x_0 \in A$ admits a neighborhood $U \subset M$ and smooth functions 
$$f_1,\dots,f_l, g_{l+1}, \dots,g_k \in C^\infty(U), \quad 0\le l\le k \le n,$$ 
such that
\begin{itemize}
\item $f_1(x_0) = \dots = g_k(x_0) = 0$
\item $d_{x_0}f_1,\dots, d_{x_0} g_k$ are linearly independent
\item $A \cap U = \bigcap_{i=1}^l f_i\inv[0,\infty) \cap \bigcap_{j=l+1}^k g_j\inv(0)$
\end{itemize}

Denote by $A_{n-k}$ the set of points $x_0$ with these properties.

 We denote the class of all such objects  by $\pol=\pol(M)$, and the subclass of compact smooth polyhedra by $\pol_c=\pol_c(M)$. 
\end{definition}

Equivalently, $A$ is a smooth polyhedron if for each $x_0$ there is such $U$ such that the pair $(U,U\cap A)$ is diffeomorphic to $(\Rn, [0,\infty)^l \times \R^{n-k+l} \times \{0\})$, i.e. to an orthant, possibly with positive codimension, within $\Rn$.
Clearly any $A\in \pol$ admits a natural {\bf stratification} $A= \bigsqcup_{l=0}^n A_l$, with the $A_l$ defined as above. Observe that $A_n$ is the  interior of $A$,  and that  each $A_l$ is a smooth locally closed submanifold of $M$ of dimension $l$.

Any $A \in \pol$ is {\bf semiconvex}, i.e. its image in $\Rn$ under any local coordinate chart has locally positive reach in the sense of \cite{cm}. The following (slightly adapted) notions from \cite {cm} therefore apply, but in simplified form due to the additional smoothness of the sets we are considering.
Given $x \in A$, the {\bf tangent cone} to $A$ at $x$ is the closed convex cone $\Tan_xA$ consisting of all $\gamma'(0)$ where $\gamma:[0,\infty)\to A$ is a smooth curve with $\gamma(0) = x$. The {\bf conormal cone} to $A$ at $x$ is the convex cone $\Nor_xA$ dual to $\Tan_xA$, i.e. 
\begin{align*}
\Nor_x(A)&:= \{\bar \xi \in T_x^*M: \langle \bar \xi, v\rangle \le 0 \text{ for all } v \in \Tan_xA\}, \\
\intertext{and put }
\Nor(A) &:= \bigcup_x \Nor_x(A) \subset T^*M
\end{align*}
Put also
\begin{align*}
N_x(A)&:= \{[\bar \xi]: 0\ne \bar \xi \in \Nor_x(A)\} \subset S^*M\\
N(A) &:= \bigcup_x N_x(A) \subset S^*M
\end{align*}
Thus $N_x(A) = \emptyset$ when $x $ lies in the interior of $A$. It is clear that if $\varphi:M\to M$ is a diffeomorphism then $(\varphi\inv)^*, \tilde \varphi$ map $\Nor(A),N(A)$ bijectively to $\Nor(\varphi(A)),N(\varphi(A))$ respectively.

The set $\Nor(A)$ is a  piecewise smooth  conic Lagrangian submanifold of $T^*M$, 
comprised of smooth submanifolds with corners, with pairwise disjoint interiors, of the conormal bundles $\Nor(A_k)$ of the strata $A_k$ of $A$. This is clear if $A$ is an orthant in $M=\Rn$, hence true for general smooth polyhedra $A$ by diffeomorphism invariance. A point $\bar \xi \in \Nor(A)$ belonging to one of these interiors is called a {\bf smooth point} of $\Nor(A)$.
Selecting an auxiliary Riemannian metric, the exponential map induces a piecewise smooth homeomorphism of the submanifold of all elements of $\Nor(A)$  of sufficiently small length to a tubular neighborhood of $A$ in $M$. We endow $\Nor(A)$ with the orientation thus induced by that of $M$. 
 
Likewise $N(A)$ is a compact piecewise smooth {\bf Legendrian} submanifold without boundary of $S^*M$ (i.e.
the contact form $\alpha$ vanishes identically on all tangent spaces of $N(A)$), and decomposes as a union of smooth submanifolds with corners of the $N(A_k)$ with disjoint interiors. We orient $N(A)$ via its identification with the boundary of the domain in  $\Nor(A)$ consisting of  of covectors of length $\le 1$  with respect to some Riemannian metric.
Thus, under the equivalence induced by the exponential map above, the conormal cycle $N(A)$ may be identified with the image of the boundary of the tubular neighborhood, oriented accordingly.
 We observe
\begin{lemma}\label{lem:oriented n}
The orientation of $N(A)$ agrees with the orientations of each $N(A_k)$, defined in the same manner, where they overlap. 
We have the equality of oriented submanifolds
\begin{equation}\label{eq:norient}
\Nor(A)-\z = \proj \inv N(A).
\end{equation}
\end{lemma}
\begin{proof} The first assertion is immediate from the identification above of the normal cycle with the boundary of a tubular neighborhood.

To prove the second assertion, identifying $S^*M$ with a submanifold of $T^*M$ with the aid of a Riemannian metric, the diffeomorphism $(r,\xi)\mapsto r\xi$ is an orientation-preserving diffeomorphism $(0,\infty) \times N(A) \to\Nor(A)-\z$, where $(0,\infty)$ is oriented in the standard way. Thus the map $(\xi,r) \mapsto r\xi$ maps $ N(A)\times (0,\infty)\to\Nor(A)-\z$ with orientation $(-1)^{n-1}$. By \ref{subsect:preimages} and \ref{sect:norient}, this yields \eqref{eq:norient}.
\end{proof}

%Note that the respective restrictions of the maps $(\varphi\inv)^*, \tilde \varphi:\Nor(A),N(A)\to \Nor(\varphi(A)),N(\varphi(A))$ are orientation-preserving, regardless of whether $\varphi$ is or is not orientation-preserving in its own right.

\begin{definition} We say that $A,B \in \pol(M)$ {\bf intersect transversely} if the strata of $A,B$ intersect pairwise transversely.
\end{definition}

\begin{lemma}\label{lem:transverse intersect} If $A,B \in \pol(M)$ intersect transversely then $A\cap B \in \pol(M)$. In this case, for all $x \in A\cap B$
\begin{equation}\label{eq:sum of Nor}
\Nor_x(A\cap B) = \Nor_x(A) + \Nor_x(B)
\end{equation}
(Minkowski sum)
and $N_x(A\cap B)$ is the union of all segments $\overline{\xi,\eta}$ for $\xi \in N_x(A), \eta\in N_x(B)$. The interiors of these segments are pairwise disjoint.
\end{lemma}
\begin{proof} The first statement is clear from the definition of smooth polyhedra.

The relation \eqref{eq:sum of Nor} is a straightforward adaptation of \cite{cm}, Theorem 4.10 (3). This statement translates directly to give the second assertion. Since by transversality $\Nor(A) \cap \Nor(B) \subset \z$, the final assertion is clear.
\end{proof}

\subsection{Morse theory via intersections} 
Let $f$ be a Morse function on $M$, i.e. $f \in C^\infty(M)$ and the graph of  $df$ meets the zero section transversely in $T^*M$. Let $A \in \pol_c(M)$. 
Put 
$$
s:S^*M\to S^*M, \quad s(\xi) = -\xi
$$
for the fiberwise antipodal map. 
We will say that $f$ is {\bf Morse on $A$} if the graph $\Gamma \subset T^*M$ of $df$ intersects $s(\Nor(A))$ only in smooth points, and every such intersection is transverse. This implies in particular that  no critical points of $f$ lie on the boundary of $A$. We orient $\Gamma$ so that its projection to $M$ is orientation-preserving. 

It is elementary (cf. e.g. \cite{guil-pol}) that the ordinary Morse condition on $f$ is equivalent to the transversality of $\Gamma \subset T^*M$ and the zero section, and that the multiplicity of the intersection at a point $x \in M$ is $(-1)^{\sigma(f,x)}$ where $\sigma(f,x)$ is the Morse index of $f$ at $x$. Morse theory thus implies that if $f$ is proper, bounded below, and admits only finitely many critical points then
$$
\z\bullet \Gamma = \chi(M).
$$
The next lemma adapts this identity to a function that is Morse on  $A \in \pol_c(M)$.

\begin{lemma}\label{lem:morse1} If $f$ is Morse on $A$ then
$$
\chi(A) =\#( s(\Nor(A))\bullet \Gamma ).
$$
More generally, if $\Gamma \cap s(\Nor(A) ) \cap \pi \inv f\inv(t)=\emptyset$ (i.e. $t$ is a regular value of $\restrict f A$ ) then
\begin{align*}
\chi(A \cap f \inv(-\infty, t]) &= \#(s(\Nor(A)) \bullet (\Gamma\cap \pi\inv f\inv(-\infty,t]) )\\
&= (-1)^n \#( (\Gamma\cap \pi\inv f\inv(-\infty,t]))\bullet s(\Nor(A)).
\end{align*}
\end{lemma}
\begin{proof} This may be proved using either the elementary approach of \cite{fucm} or the more sophisticated and general theory of \cite{gor-macp}.
\end{proof}

We express this fact in terms of the cosphere bundle $S^*M$.  Put $[\Gamma]\subset S^*M$ for the image of $\Gamma -\z$ under projectivization.
 
%  Since the latter is the choice that is compatible with the 
% 
% Alternatively, if $\frac{\partial}{\partial r},w_1,\dots,w_{2n-1}$ constitute a positively oriented basis of $T_{\bar \xi}(T^*M)$ at $\bar \xi \ne 0$, where the first element  again denotes the Euler field, then the images of 
%$w_1,\dots,w_{2n-1}$ under the differential of the projectivization map constitute a positively oriented basis of $T_{[\bar \xi]}(S^*M)$. Note that this orientation is $(-1)^n$ times the orientation of $S^*M$ inherited from  its local identifications with open subsets of $\Rn \times S^{n-1}$.
%
%The orientation of $S^*M$ has been chosen so that Lemma \ref{lem:morse1} translates directly into the following. 

\begin{lemma}\label{lem:morse2} If $f $ is Morse on $A\in \pol_c$ then the set of critical points of $\restrict f A$ is finite, and for every regular value $t$ of $\restrict fA$
\begin{align*}
\chi(A\cap f\inv(-\infty,t]) &=  (-1)^n \#(([\Gamma]\cap \pi\inv f\inv(-\infty,t])\bullet s(N(A)))\\
&\quad\quad  + \sum_{\substack {x \in A \\ d_xf=0\\f(x) \le t}} (-1)^{\sigma(f,x)}.
\end{align*}
\end{lemma}
\begin{proof} This follows at once from \eqref{eq:intersect project} and Lemma \ref{lem:oriented n}.
\end{proof}

Observe that the intersection product that occurs in Lemma \ref{lem:morse2} is commutative, since  the dimensions of the factors are $n-1$ and $n$ respectively (i.e. one of them is even).

\section{Smooth valuations}  Following \cite{ale06, ale-be09}, a {\bf smooth valuation} on $M$ is a set function $\mu:\pol_c(M)\to \R$ that is expressible as
\begin{equation}\label{eq:def valuation}
\mu(A) = \int_{N(A)} \beta + \int_A\gamma
\end{equation}
for some smooth differential forms $\beta \in \Omega^{n-1}(M), \gamma \in \Omega^n(S^*M)$.
We denote by $\V(M)$ the vector space of all smooth valuations on $M$. The classic works \cite{chern1, chern2} imply that the Euler characteristic $\chi \in \V(M)$.

A related notion is that of a (smooth) {\bf curvature measure} on $M$, defined as follows. Put $\m(M)$ for the space of all signed Radon measures on $M$. A curvature measure on $M$ is a set function $\Phi: \pol(M) \to \m(M)$ given by
\begin{equation}\label{eq:def cm}
\Phi(A,E) := \Phi(A)(E) := \int_{N(A)\cap \piinv E}\beta +  \int_{A\cap E}\gamma.
\end{equation}
Alternatively we may think of such $\Phi$ as a map  $C^\infty(M)\to\V(M)$, associating to a smooth function $f$ the smooth valuation
\begin{equation}
\Phi_f:A \mapsto \Phi(A,f):=  \int_{N(A)}\pi^* f \cdot\beta + \int_{A}f \cdot\gamma
\end{equation}

We denote by $\C(M)$ the vector space of all smooth curvature measures on $M$. Clearly there are surjective linear maps
\begin{equation}\label{eq:projections}
\Omega^{n-1}(S^*M)  \oplus \Omega^n(M)\to \C(M) \to \V(M).
\end{equation}
We put 
\begin{equation}
\meas{\beta,\gamma}\in \C(M), \quad \cur{\beta,\gamma} \in \V(M)
\end{equation}
for the respective images of $(\beta,\gamma)$ under the first map and the composition of the two maps.

A smooth valuation $\mu \in \V(M)$ is finitely additive in the sense that if $A,B,A\cap B, A\cup B \in \pol_c(M)$ then
$$
\mu(A\cup B) =\mu(A) +\mu(B) - \mu(A\cap B).
$$
This follows from the corresponding relation among the normal cycles:
$$
\int_{N(A\cup B)} =\int_{N(A) }+\int_{N(B)} - \int_{N(A\cap B)}.
$$
We will not make use of this relation except to argue that, since a smooth polyhedron may be decomposed as finely as desired, a smooth valuation is determined by its values on subsets of small open sets in $M$.

\subsection{The variation and the point function of a smooth valuation}\label{sect:variation}
There are canonical maps 
$$
\mathcal F:\V(M) \to C^\infty (M),\quad \Delta:\V(M) \to \Omega^{n}(S^*M)
$$
determined by the relations
\begin{align}
\label{eq:fmu 1}\mathcal F_\mu (x) &:= \mu(\{x\}),\\
\label{eq:fmu 2}\alpha \wedge \Delta_\mu
%= d\Delta_\mu
&=0,\\
\label{eq:fmu 3}\ddt \mu(F_t(A)) &= \int_{N(A)}   v\without  \Delta_\mu
\end{align}
for any $A \in \pol_c(M)$ and any smooth vector field $v$ on $M$, where $F_t$ is the flow of $v$. Note that we have abused notation in the integrand on the right hand side: formally, we should replace $v$ by a lift of $v$ to $S^*M$ and observe that since $\Delta_\mu$ is a  
%closed 
multiple of the contact form  and $N(A)$ is 
%closed and 
Legendrian, the value of the integral is independent of choices.

\begin{proposition}\label{prop:unique} \begin{enumerate}
\item\label{item:uniqueness} The maps $\mathcal F, \Delta$ are uniquely determined by the properties \eqref{eq:fmu 1}, \eqref{eq:fmu 2}, \eqref{eq:fmu 3}.
\item\label{item:kernel thm} If $\mathcal F_\mu = \Delta_\mu = 0$ then $\mu =0$. 
\item\label{item:ker2} If $\mu =\cur{\beta,\gamma}$, then 
\begin{align}
\mathcal F_\mu &= \pi_* \beta,\\
\label{eq:structure eqns} \Delta_\mu &= D\beta+\pi^*\gamma,
\end{align}
where $D$ denotes the Rumin differential \cite{rumin}.$ \quad \square$
\end{enumerate}
\end{proposition}

\noindent{\bf Remarks.} 
\begin{enumerate} 
\item Assertions \eqref{item:kernel thm} and \eqref{item:ker2} encompass the Kernel Theorem of Bernig and Br\"ocker \cite{bebr07}, and  may be proved as follows. By finite additivity it is enough to show that $\mu(A) =0$ in the case where $M$ is a convex open set $U \subset\Rn$ containing $0$ and $A\subset U$ is convex. If $\Delta_\mu =0$ then taking $v$ to be the Euler vector field we find that $\mu(tA) = \mu(A)$ for $A \in \pol_c(U)$ and $0<t\le 1$. On the other hand, in terms of weak convergence of currents $N(tA) \to N(\{0\})$ as $t \downarrow 0$, so that $\mu(A) = \mathcal F_\mu(0)$.

\item The map $\mathcal F$ is surjective: given $f \in C^\infty (M)$ and representing the Euler characteristic $\chi$ as $\cur{\beta,\gamma}$, clearly $\mathcal F_{\lcur\pi^*f \cdot \beta,\gamma\rcur} =f$.  Furthermore, by \cite{rumin} and \eqref{eq:structure eqns}, the image of $\Delta$  consists precisely of the $n$-forms of the form $\pi^*\gamma + \omega, \gamma \in \Omega^n(M), \omega \in \Omega^n(S^*M)$, where $\omega $ is exact and $\alpha \wedge \omega=0$. Since the long exact Gysin sequence (cf. e.g. \cite{milnor}) for the cosphere bundle includes the segment
$$\begin{CD}
 H^n(M)@>{\pi^*}>> H^n(S^*M) @>{\pi_*}>> H^1(M) 
\end{CD}
$$
it follows that this image may also be characterized as the space of such $\pi^*\gamma + \omega$ for which $\omega $ is closed, $\alpha \wedge \omega=0$, and $\pi_* \omega \in \Omega^1(M)$ is exact.

 One may interpret this last condition from the valuation-theoretic perspective as follows. Given such $\gamma,\omega$ we wish to construct $\mu \in \V(M)$ with $\Delta_\mu =\pi^*\gamma +\omega$. Since $\Delta_{\cur{0,\gamma}} =\pi^*\gamma$, we may omit the first term. By finite additivity it is enough to give the value $\mu(A)$ in the case where $A$ is contractible. Replacing $\Delta_\mu$ by  $\omega$  in the variation formula and contracting $A$ to a point as above, we see that $\mu$ is well-defined if it is well-defined on singletons, i.e. iff $\mathcal F_\mu$ is well-defined; or in other words if given two smooth paths $\sigma_1,\sigma_2$ joining arbitrarily chosen points $x,y \in M$ there exist respective local primitives $\beta_i$ for $\omega$ along $\sigma_i$ such that $\pi_*\beta_1(x) - \pi_*\beta_1(y) = \pi_*\beta_2(x) - \pi_*\beta_2(y)$. By Stokes' theorem this is equivalent to the vanishing of $\pi_*\omega $ on the 1-cycle ${\sigma_1-\sigma_2}$.

\item  The relation \eqref{eq:structure eqns} implies that $\Delta_\mu$ is closed. It would be desirable to have a valuation-theoretic proof of this fact similar to the argument above, along the following (incomplete) lines.

Let $A,B \in \pol_c(M)$ and let $v,w$ be vector fields with flows $F,G$ respectively, such that $F_1(A) = G_1(A) = B$. Define the $n$-chains $\hat F: = F(N(A) \times[0,1]),\hat G:= G(N(A) \times[0,1])\subset S^*M$. Then by \eqref{eq:fmu 3} and the coarea formula
\begin{align*}
\int_{\hat F} \Delta_\mu&= \int_0^1 \int_{N(F_t(A))} v \without \Delta_{\mu} \, dt \\
&= \mu(B)- \mu(A) \\
&=\int_{\hat G} \Delta_\mu.
\end{align*}
A homotopy between $F,G$ then determines an $(n+1)$-chain $H$ with $\partial H = \hat F -\hat G$, so that by Stokes' theorem
$$
\int_H d\Delta_\mu = \left(\int_{\hat F}- \int_{\hat G}\right) \Delta_\mu = 0.
$$

\item As a side point, we recall that Rumin \cite{rumin} showed that every degree $n$ cohomology class of any $(2n-1)$-dimensional contact manifold contains a representative belonging to the contact ideal.
\end{enumerate}

\section{The main construction} 
The main theater is the oriented smooth $(3n-1)$-dimensional manifold
 \begin{equation}
 \Sigma:= \{(\xi,\zeta,\eta) \in (S^*M)^3: \pi\xi=\pi\eta = \pi \zeta,  \xi \ne \pm\eta,\zeta \in \overline{\xi,\eta} \},
 \end{equation}
 where $\overline{\xi,\eta}$ denotes the  open segment in $S^*_{\pi \xi}M$ joining $\xi,\eta$.
 Thus there are three associated projections 
 $$
 \xi, \zeta,\eta:\Sigma\to S^*M.
 $$
% as in the diagram
% \begin{equation}\label{diag:alebe}
%\begin{CD} 
% \Sigma@>\zeta>> S^*M\\
%@V(\xi,\eta)VV \\
%S^*M \times_M S^*M 
%\end{CD}
%\end{equation}
We view $\Sigma$ as the total space of a
 bundle over $M$ with fiber $F_x$ over $x\in M$ diffeomorphic to $S_xM \times (0,\frac \pi 2) \times S_xM$ with the diagonal and antidiagonal deleted, via the map 
 \begin{equation}\label{eq:explicit zeta}
 (\xi,t,\eta) \mapsto (\xi,\zeta:= [\cos t\,\bar\xi + \sin t \,\bar \eta],\eta) ,
 \end{equation} 
 where $\bar \xi, \bar \eta \in T^*_xM-\z$ are arbitrarily chosen representatives of $\xi = [\bar \xi],\eta = [\bar \eta]$, giving rise to an  open embedding
$$
\iota :\Sigma\to S^*M \times \left(0,\frac \pi 2\right) \times_M S^*M,
$$  
canonical up to orientation-preserving reparametrizations of $(0,\frac \pi 2)$.
Following our usual practice  we orient $\Sigma$ in terms of the local product structure $\Rn \times S^{n-1}\times (0,\frac \pi 2) \times S^{n-1}$.

\subsection{A completion of $\Sigma$}\label{sect:completion}
In the constructions of Theorems \ref{thm:ab}  and \ref{thm:well defined} below we will perform a number of fiber integrals with respect to the natural projections $\xi,\eta,\zeta$. On the face of it, these fiber integrals are problematic since these maps are not proper. However, this obstacle may be circumvented by working implicitly over a convenient completion $\tilde \Sigma$ with the following properties:
\begin{enumerate}
\item $\tilde \Sigma$ is  a smooth oriented manifold with corners, equipped with  a smooth proper map $\sigma :
\tilde \Sigma \to S^*M\times_M S^*M \times_M S^*M:= \{(\xi,\zeta,\eta)\in (S^*M)^3: \pi \xi = \pi\zeta = \pi\eta\}$,
%\item the image of $\tilde \Sigma$ under $\sigma$ is equal to the closure of $\Sigma$,
\item   the restriction of $\sigma $ to the interior of $\tilde \Sigma $ yields a   diffeomorphism with $\Sigma$,
\item each of the maps $\xi\circ \sigma, \eta\circ \sigma,\zeta\circ \sigma$ is the projection of a smooth fiber bundle over $S^*M$ with fiber equal to a smooth oriented compact manifold with corners.
\item \label{item:corner fibers} the map $(\xi \circ \sigma ,\eta\circ \sigma): \tilde \Sigma \to S^*M \times_M S^*M$ is the projection of a smooth fiber bundle with fiber given by a smooth compact oriented manifold with corners.
\end{enumerate}

We construct $\tilde \Sigma$ as the Cartesian product $\tilde S \times [0,1]$, where $\tilde S$ is the oriented blowup of $S^*M\times_M S^*M$ over the union of the diagonal and the antidiagonal. Introducing for convenience an auxiliary Riemannian metric, $\tilde S$ may  be realized as a fiber bundle over $M$ with fiber over $x$ equal to $\mathbb U(S_xM) \times [0,\pi ]$, where $\mathbb U(S_xM)= \{(\xi,\tau):\langle \xi,\tau\rangle = 0\}\subset S_xM \times S_xM$ is the unit tangent bundle of the unit cotangent sphere $S_xM$. 
The map $\tilde \Sigma \to S^*M\times _M S^*M$ of \eqref {item:corner fibers} is then realized fiberwise by
$$
(\xi, \tau, t) \mapsto (\xi,\cos t \,\xi + \sin t \, \tau) \in S_x M \times S_xM.
$$
The maps $\xi \circ \sigma,\eta\circ \sigma$ are the compositions of this map with the projections to the first and second factors, respectively. 
 The map $\zeta \circ \sigma$ is then 
$$
(\xi, \tau, t, r) \mapsto \cos ( r\,t) \,\xi + \sin (rt) \, \tau\in S_xM, \quad r \in [0,1],
$$
The blowdown map $\sigma$ is obtained by assembling the three maps just described.

\subsection{The normal cycle of a transverse intersection} The geometric meaning of $\Sigma$ lies in the following.
%For this it will be helpful to recall in detail how the orientation of the normal cycle of a submanifold $V \subset M$ is determined, with the help of the Riemannian structure we have introduced. We may then think of $N(V)$ as the submanifold of  the conormal vector bundle $\Nor(V)\subset T^*M$ consisting of all covectors of length 1. Identifying $\Nor(V)$ with the normal vector bundle $\subset TM$ via the Riemannian metric, the orientation of $\Nor(V)$ is determined by the condition that the exponential map $\Nor(V)\to M$ is an orientation-preserving diffeomorphism in the neighborhood of the zero section. Now $N(V)$ is oriented as the boundary of the submanifold of (co)vectors of length $\le 1$, i.e. if $\xi \in N(V)$ and the  derivative of the length function in the direction $v\in T_\xi \Nor(V)$ is positive, then an ordered basis $w_1,\dots,w_{n-1}$ for $T_\xi N(V)$ is positively oriented iff $v,w_1,\dots,w_{n-1}$ is a positively oriented basis for $T_\xi \Nor(V)$.

\begin{lemma}\label{lem:transverse intersect} If $A,B \in \pol( M)$ intersect transversely then $\xi^{-1}N(A)$ and $\eta\inv N(B)$ intersect transversely as submanifolds of $\Sigma$. The restriction of $\zeta$ to the intersection yields a diffeomorphism 
$$
\xi^{-1}N(A)\bullet\eta\inv N(B)\simeq N(A\cap B)-\left[(N(A) \cap \pi\inv B )\cup (N(B) \cap \pi\inv A)\right]
$$ 
of parity $(-1)^n$.
\end{lemma}
\begin{proof} Recall that the transversality of $A,B$ means precisely that the various strata of $A,B$ meet pairwise transversely. Thus by  Lemma \ref{lem:oriented n} the present assertion follows from the special case in which $A,B$ are properly embedded submanifolds (viz. strata) that intersect transversely. By diffeomorphism invariance we may even assume that  $M=\Rn$, endowed with the standard euclidean metric, and that
\begin{align*}
A&= \Rk\times 0\times \Rm, \\
B&= \Rk\times \Rl\times 0,
\end{align*}
$ \ k + l + m = n,$ are transverse coordinate planes.

Denote by $S, S_A\subset S, S_B\subset S$ the spheres of $\Rn,0\times \Rl \times 0, 0\times 0 \times \Rm$ respectively, with their standard orientations. 
Thus we have the equalities of oriented manifolds
\begin{align*}
S\Rn &= \Rn \times S,\\
\Sigma&= \Rn \times S \times \I \times S,\\
N(A) & = (-1)^{k + (l-1)m}(\Rk\times 0\times\Rm)\times S_A\\
N(B) & =  (-1)^{k+l}(\Rk\times  \Rl \times 0) \times S_B.
\end{align*}
Here the numerical factors signify corresponding changes of orientation, induced by the requirement (stated in the paragraph preceding Lemma \ref{lem:oriented n}) that the map $(p, q) \mapsto p+q$, restricted to the normal cycle of $X$, gives an orientation-preserving diffeomorphism to the boundary of a tubular neighborhood of $X$. Here we also take into account the fundamental relation
$$
\partial (C\times D) = \partial C \times D + (-1)^{\dim C} C \times \partial D.
$$
Then
 \begin{align*}
\xi\inv N(A) &=(-1)^{k + (l-1)m}( \Rk\times 0\times\Rm) \times  S_A \times \I \times S,\\
\eta\inv N(B) &=(-1)^{ k+l+n(m-1)}(\Rk\times  \Rl \times 0) \times S \times \I \times S_B \\
\end{align*}
as oriented manifolds, following the convention of Section \ref{subsect:preimages},
with unoriented intersection 
\begin{equation}\label{eq:prodint}
\xi\inv N(A)\cap\eta\inv N(B)= ( \Rk\times 0 \times 0) \times S_A \times \I \times S_B.
\end{equation}
The image under $\zeta$ of the oriented product represented by the right hand side of \eqref{eq:prodint} is clearly equal to $(-1)^kN(A\cap B)$, with negligible sets (smooth submanifolds of positive codimension) deleted.
Thus it remains to show that 
\begin{align}
\notag[( \Rk\times 0 \times \Rm)\times  S_A &\times \I \times S ]\bullet[( \Rk\times \Rl \times 0) \times S \times \I \times S_B] 
\\
\label{eq:prodint2} &= (-1)^{mk + l + k}( \Rk\times 0 \times 0) \times S_A \times \I \times S_B.
\end{align}
Here the parity on the right is the solution $x$ of the mod 2 congruence
$$
(k+ (l-1)m) + (k + l + n(m-1)) + x \equiv n + k.
$$

Given $v \in S_A, w \in S_B$, we have the equalities of oriented tangent spaces
\begin{align*}
T_v S & = (-1)^k \Rk\oplus T_A \oplus \Rm,\\
T_w S & = (-1)^{k+l}\Rk\oplus \Rl \oplus T_B,
\end{align*}
where we abbreviate $T_A:= T_vS_A, \ T_B:= T_wS_B$, with $\dim T_A = l-1, \dim T_B = m-1$. Referring to the left hand side of \eqref {eq:prodint2}, we thus find that
 at a representative point $(0,v,t,w)\in ( \Rk\times 0 \times 0) \times S_A \times \I \times S_B$ the tangent spaces to the first factor, the intersection, the second factor, and the total space are the respective oriented direct sums
\begin{align*}
(-1)^{k+l}(\Rk\oplus 0 \oplus \Rm  ) \oplus( 0\oplus T_A \oplus 0) &\oplus \R \oplus  (\Rk\oplus \Rl \oplus T_B),
\\
(\Rk\oplus 0\oplus 0 ) \oplus( 0\oplus T_A \oplus 0) &\oplus \R \oplus  (0\oplus 0\oplus T_B),
\\
(-1)^{k}(\Rk\oplus \Rl \oplus 0 ) \oplus( \Rk\oplus T_A \oplus \Rm) &\oplus \R \oplus  (0\oplus 0\oplus T_B),
\\
(-1)^{l}(\Rk\oplus \Rl \oplus \Rm ) \oplus( \Rk\oplus T_A \oplus \Rm) &\oplus \R \oplus  (\Rk\oplus \Rl \oplus T_B).
\end{align*}
Employing the orientation convention of Section \ref{sect:orient intersect} with care and attention, this yields the stated conclusion.
\end{proof}

\noindent {\bf Remark.} There is a corresponding, and somewhat simpler, calculation that takes place in the full cotangent bundle, replacing $S^*M$ by $T^*M$, $N$ by $\Nor$ and $\zeta: \Rn \times S^{n-1}\times \I \times S^{n-1} \to S\Rn$ by the map
$$
\Rn \times \Rn \times \Rn \to \Rn \times \Rn, \quad (x,v,w) \mapsto (x, v+w).
$$
Lemma \ref{lem:transverse intersect} then follows from a functoriality argument using Lemma \ref{lem:oriented n}. Although this approach is in some ways more convincing, we have opted for the present approach since  the functoriality construction seems a bit too fancy here.

\subsection{The Alesker-Bernig formula} Alesker-Bernig \cite{ale-be09} showed that the product of valuations may be expressed in terms of their underlying differential forms as in Theorem \ref{thm:ab} below. For our purposes we may take their formula, given below, as the definition of the product of a valuation and a curvature measure.

 \begin{theorem}[Alesker-Bernig \cite{ale-be09}, Theorem 2]\label{thm:ab} Let 
 $$
 \mu \in \V(M), \ \beta \in \Omega^{n-1}(S^*M), \ \gamma \in \Omega^n(M).
 $$ Then 
\begin{equation}\label{eq:ab beta gamma}
\mu \cdot [{\beta,\gamma}] = [\theta,\psi],
\end{equation}
 where
\begin{align}
\label{eq:alb1}\theta &= \xi_*(  \zeta^* \beta\wedge \eta^*s^*\Delta _\mu) + \pi^* \mathcal F_\mu \cdot \beta, \\
\label{eq:alb2}\psi&= \pi_*( \beta \wedge s^*\Delta_ \mu ) + \mathcal F_\mu \cdot \gamma .
\end{align}
\end{theorem}

\noindent{\bf Remarks.} Replacing the forms $\beta,\gamma,\theta,\psi$ by their multiples by smooth functions $f$, \eqref{eq:ab beta gamma} is equivalent to the same formula with the curvature measures $[\cdot,\cdot]$ replaced by valuations $\lcur\cdot,\cdot \rcur$.

The key term here is the first summand on the right hand side of \eqref{eq:alb1}. Formally, this expression differs from that of the original source \cite{ale-be09} in two respects: there the antipodal map $s$  is not involved, and our space $\Sigma$, consisting of all $(\xi, \zeta, \eta) $ such that $\zeta $ lies on the segment joining $\xi$ to $\eta$, is replaced there by the manifold of all such triples such that $\xi$ lies on the segment joining $\eta$ to $\zeta$.  Thus the two discrepancies cancel each other, since $\xi,\zeta,\eta$ satisfy the  first condition iff $\xi, s(\eta),\zeta$ satisfy the second.

Another cosmetic difference between our formulas and those of \cite{ale-be09} is the presence of the term $\pi^* \gamma$ in the second factor of the first term of \eqref{eq:alb1}. It is easy  to see that $\xi_*(  \zeta^* \beta\wedge \eta^*\pi^* \gamma) =0$ for dimensional reasons.

\section{Kinematic valuations} \label{sect:kin val}

We arrive finally at the main construction of this paper.

\begin{definition}\label{def:kin val}  Let $M$ be a smooth oriented manifold.
\begin{enumerate}
\item \label{item:admissible} An {\bf admissible measured family} of diffeomorphisms of $M$ is a triple $(P,dp, \varphi)$ where \begin{itemize}
\item $P$ is a smooth oriented manifold.
\item $\varphi :P\times M\to M$ is a smooth map such that each $\varphi_p:= \varphi(p,\cdot)$ is an orientation-preserving diffeomorphism of $M$.
\item The restriction to $\supp dp \times M$ of the map $(p,x) \mapsto (\varphi_p(x), x)$ is proper.
\item Put $\tilde \varphi:P \times S^*M \to S^*M$ for the induced family of contact diffeomorphisms of the cosphere bundle. Then  for each $\xi \in S^*M$ the map $p\mapsto \tilde \varphi(p,\xi)$ is a submersion.
\end{itemize}
\item A {\bf kinematic valuation} on a smooth oriented manifold $M$ is a set function $\nu= \nu(\mu,X,P,dp,\varphi)$ determined by the following data:
\begin{itemize}
\item a smooth valuation $\mu \in \V(M)$,
\item a compact smooth polyhedron $X \in \pol_c(M)$,
\item  an admissible measured family $(P, dp, \varphi)$ of diffeomorphisms of $M$.
\end{itemize} 
We then put

\begin{equation}\label{eq:alt def}
\nu(A):= \int_P \mu(A\cap \varphi_p(X)) \, dp.
\end{equation}

\item If $\mu = \chi$, then $\nu$ is a {\bf principal kinematic valuation}.
\end{enumerate}
\end{definition} 

It is clear that if $(M,G)$ is Riemannian isotropic in the sense described in the Introduction, and $dg$ is  Haar measure on $G$, then $(M,dg, G)$ is an admissible measured family of diffeomorphisms of $M$.
 
%\noindent {\bf Remark.} The historical origin of this notion is as follows. Let $M$ be an oriented Riemannian manifold equipped with a group $G$ of isometries that acts transitively on the sphere bundle $S^*M$ (i.e. $(M,G)$ is a {\it Riemannian isotropic space}). Taking $P=G$ and $dp=$ Haar measure on $G$, then any $\mu,X$ as above determine a kinematic valuation $\nu$ on $M$. If $M$ is $G$-invariant then $\nu$ is itself $G$-invariant. These facts are proven in \cite{fu90}.
 
\subsection{Kinematic valuations are smooth}\label{sect:smooth} A priori it is not clear under what conditions $\nu(A)$ is well-defined. We show now that this is the case whenever $A \in \pol_c(M)$, and that $\nu$ is in fact given by an element of $\V(M)$ in this case. To prove this we express $\nu$ explicitly as $\lcur{\theta,\psi}\rcur$ using the constructions above.
% Afterward we will establish Theorem \ref{thm:main} by showing that this expression constitutes an instance of the Alesker-Bernig formula Theorem \ref{thm:ab}. 

Observe first that by Lemma \ref{lem:basics} the $(n-1)$-dimensional current 
$$
\beta \mapsto \int_P \left(\int_{\tilde \varphi_{p}( N(X))}\beta\right) \ dp = \int_P \left(\int_{N(\varphi_{p}( X))}\beta \right)\ dp 
$$
is smooth. Let $\omega \in \Omega^n (S^*M)$ denote the associated differential form. Since $N(X)$ is closed and Legendrian it follows that
$$
d \omega = \alpha \wedge \omega =0.
$$
Define also $f \in C^\infty(M)$ by
$$
f(x) :=  dp(\{p\in P: x \in \varphi_p(X)\}).
$$

\begin{theorem}\label{thm:well defined} Let $(P,dp,\varphi)$ be an admissible measured family of diffeomorphisms of $M$, and $X \subset M$ a compact smooth polyhedron. Then
\begin{enumerate}
\item \label{item:well defined 1}If $A \in \pol_c(M)$ then $A$ and $\varphi_p(X)$ intersect transversely for a.e. $p \in P$. 
\item \label{item:well defined 2} The integral on the right hand side of \eqref{eq:alt def} converges absolutely.
\item \label{item:well defined 3} The resulting set function $\nu$ is a smooth valuation.
More precisely, if  $\mu =\lcur \beta,\gamma \rcur$  then $\nu = \lcur \theta,\psi \rcur$, where
\begin{align}
\label{eq:smexp1}\theta &:= (-1)^n \xi_*(  \zeta^* \beta\wedge \eta^*\omega) + \pi^*f \cdot \beta, \\
\label{eq:smexp2}\psi&:=  \pi_*(  \beta\wedge \omega ) + f\cdot \gamma ,
\end{align}
with $\omega, f$ defined as above.
\end{enumerate}
\end{theorem}
 
 \begin{proof} Conclusion \eqref{item:well defined 1} follows at once from
conclusion \eqref{item:basic1} of Lemma \ref{lem:basics}.

\eqref{item:well defined 2}: If $A $ and $\varphi_p(X)$ intersect transversely then clearly $N(A) \cap N(\varphi_p(X) ) = N(A) \cap sN(\varphi_p(X) )=\emptyset$. Thus we may write
$
N(A \cap \varphi_p(X))$ as the disjoint union of the three pieces
$$ N(A) \cap \pi\inv \varphi_p(X) , \ N(\varphi_p(X)) \cap \pi\inv A,  \,  N(A\cap \varphi_p(X)) - [\pi \inv A \cup \pi\inv \varphi_p(X)]
$$
and accordingly
\begin{align}
\notag\nu(A) &= \int_P \mu(A \cap \varphi_p(X))\, dp\\
\notag& = \int_P \left[\int_{N(A \cap \varphi_p(X))}\beta+ \int_{A\cap \varphi_p(X)}\gamma \right]\, dp \\
\label{eq:four terms}&=  \int_P \int_{N(A\cap \varphi_p(X)) - [\pi \inv A \cup \pi\inv \varphi_p(X)]}\beta \, \, dp + \int_P \int_{N(A) \cap \pi\inv \varphi_p(X)}\beta \, \, dp\\
\notag&  \quad  +   \int_P \int_{N(\varphi_p(X)) \cap \pi\inv A}\beta \, \, dp + \int_P\int_{A\cap \varphi_p(X)}\gamma \, dp .
\end{align}

It is clear that the second, third, and fourth of these integrals are all absolutely convergent. It remains to show that the same is true of the first. 

Recalling the construction of the space $\tilde \Sigma$ from Section \ref{sect:completion}, consider the map $\Psi:=(\xi\circ\sigma,\tilde \varphi_p\inv \circ \eta\circ \sigma):P\times \tilde \Sigma \to S^*M \times S^*M$. By condition \eqref{item:corner fibers} of Section \ref {sect:completion}, the submersivity condition on $\tilde \varphi$ implies that $\Psi$ is again a submersion, and that  the same is true of its restriction to the strata of the manifold with corners $P\times \tilde \Sigma$. In particular each such restriction is transverse to $N(A) \times N(X)$. It follows that the preimage of $N(A) \times N(X)$ is again a smooth properly embedded submanifold with corners of $P \times \tilde \Sigma$. By the properness condition on $\varphi$, the intersection of this submanifold with corners with the support of $dp$ is compact. By Lemma \ref{lem:transverse intersect}, the term in question is the integral over this set of a continuous object (viz. the product of the pullbacks of the differential form $\beta$ and the density $dp$). Thus the conclusion follows.

\eqref{item:well defined 3} We show that the four terms of \eqref{eq:four terms} equal respectively the four  terms in \eqref{eq:smexp1}, \eqref{eq:smexp2}. 

The second and fourth of these equalities, i.e.
\begin{align*}
\int_{N(A)} \pi^* f \cdot \beta &= \int_P \int_{N(A) \cap \pi\inv \varphi_p(X)}\beta \, \, dp, \\
\int_A f \cdot \gamma &= \int_P \int_{\varphi_p(X)\cap A} \gamma \, dp
\end{align*}
follow at once from the Fubini-Tonelli theorem.
The third identity follows  from \eqref{eq:intersect integrate 2}, with $\Sigma $ replaced by $S^*M$, 
$Y$  by $\pi\inv A \subset S^*M$, $X$ by $N(X)$, and $\phi$ by $\tilde \varphi$.

For the first identity, by Lemma \ref{lem:transverse intersect} the domains of integration of the inner integrals are
\begin{align*}
N(A\cap \varphi_p(X)) - [\pi \inv A \cup \pi\inv \varphi_p(X)]
&= (-1)^n\zeta(\xi\inv N(A)\bullet\eta\inv\tilde \varphi_pN(X)).
\end{align*}
 Thus by \eqref{eq:intersect integrate 3} and \eqref{eq:fiber int over z} the term in question becomes
\begin{align*}
 (-1)^n\int_P \int_{\xi\inv N(A)\bullet\eta\inv\tilde \varphi_pN(X)}\zeta^*\beta \, \, dp& =
(-1)^n \int_{\xi\inv N(A)}   \eta^* \omega\wedge \zeta^*\beta \\
 & =  (-1)^n\int_{N(A)} \xi_*(   \eta^* \omega\wedge \zeta^*\beta )
\end{align*}
as claimed.
\end{proof}

%
%{\bf Claim 1.} $ \int_P  \int_{A\cap \varphi_p(X)}\gamma \, dp = \int_A \psi_1$.
%
%This follows from the calculation
%\begin{align*}
% \int_P  \int_{A\cap \varphi_p(X)}\gamma \, dp &= \int_P\left( \int_{y = \varphi_p(x), y \in A, x \in X} \varphi^*\gamma\right) \, dp \\
%  & = \int_{\Delta_1\cap (y,\varphi)\inv(A\times X)} dp \wedge \varphi^* \gamma \\
%&= \int_{A\times X} \bar \psi_1 .
%\end{align*}
%
%Next we show how the terms $\psi_2,\theta_1,\theta_2$ all arise from the first integral above.
%Let $f,g$ be auras for $A,X$ respectively,  and $r$ a length function compatible with both. By the hypotheses on $\varphi$, for a.e. $p\in P$ the function $f + \varphi_p^* g$ is a defining function for $A \cap \varphi_p(X)$. 
%
%Consider the map
%$$
%\Phi:(P\times T^*M )\times_M T^*M  \to T^*M, \quad \Phi(p,\xi,\eta) := \xi +\tilde \phi_p(\eta)
%$$
%Then
%$$
%\Phi (\{p\} \times \partial f \times \partial g) = \partial (f + \varphi^* g)
%$$

\subsection{An alternative characterization of principal kinematic valuations }
Now assume that $\mu = \chi$, and define $f,\omega$ as  above. Following section \ref {sect:variation}, the resulting principal kinematic valuation may be characterized in terms of its variation and point function as follows.

\begin{theorem}\label{thm:nu chi} The smooth valuation $\nu:= \nu(\chi,X,P, dp,\varphi)$ is characterized by the relations
$$
\Delta_\nu =   (-1)^n s^*\omega, \quad \mathcal F_\nu =f.
$$
\end{theorem}
\begin{proof} The second relation is immediate. To prove the first it is enough to show that if $v$ is a smooth vector field with flow $F_t$ and $A \in \pol_c(M)$, then 
\begin{equation}\label{eq:variation}
 \ddtzero \nu(F_t(A)) = (-1)^n\int_{N(A)} v\without\,  s^*\omega.
\end{equation}
In fact it is enough to prove this in the case where $A$ is a smooth compact domain, and $v$ is outward pointing along the boundary of $A$ (given such $A$, any smooth vector field may be expressed as a difference of vector fields with this property). 

In this case there exists a Morse function $g \in C^\infty(M)$ and $\eps >0$  such that 
$$
F_t( A )= g\inv(-\infty,t] 
$$
whenever $|t|<\eps$, and all such $t$ are regular values of $g$.
Let $\tilde v$ denote a lift of $v$ to $S^*M$ that is tangent to the graph $[\Gamma]$ of $[dg]$ within the open set $\pi\inv g\inv(-\eps,\eps)$. Using Lemma \ref{lem:preimages} \eqref{item:pushdown intersect}, we deduce that
$$
[\Gamma ]\bullet \pi\inv g\inv(t) = N(g\inv(-\infty,t]).
$$
Thus, if $-\eps < a<b <\eps$, then by the coarea formula
\begin{equation}\label{eq:int over graph}
\int_{[\Gamma ]\bullet\pi\inv g\inv[a,b] }\gamma = \int_a^b\left( \int_{N(g\inv(-\infty,t])}  \tilde v \without \, \gamma \right)\, dt 
\end{equation}
for any $\gamma \in \Omega^n(S^*M)$.
 Thus by Lemma \ref{lem:morse2}, Lemma \ref{lem:basics} \eqref{eq:intersect integrate}, and \eqref{eq:int over graph}
\begin{align*}
  \ddtzero \nu(g\inv[0,t])&=\lim_{t\to 0} t\inv\left( \nu(g\inv(-\infty,t]) - \nu(g\inv(-\infty,0])\right)\\
  &=\lim_{t\to 0} t\inv \int_P (\chi(\varphi_p(X)\cap g\inv(-\infty,t])-\\
  & \quad\quad \chi(\varphi_p(X)\cap g\inv(-\infty,0])) \, dp\\
  &=(-1)^n\lim_{t\to 0} t\inv \int_P \#[([\Gamma] \cap \pi\inv g\inv(0,t]) \bullet sN(\varphi_p(X))] \, dp\\
   &=(-1)^n\lim_{t\to 0} t\inv \int_{[\Gamma] \cap \pi\inv g\inv(0,t] }s^*\omega \\
&= (-1)^n\int_{N(g\inv(-\infty,0]) }\tilde v\without \,s^*\omega.
\end{align*}
The desired identity now follows from conclusion \eqref{item:uniqueness} of  Proposition \ref{prop:unique}.
\end{proof}

\subsection{The main theorem}
 Let us fix the data $X,P,dp,\varphi$ and abbreviate $\nu_\mu:= \nu(\mu, X,P,dp,\varphi)$ for $\mu \in \V(M)$.
 \begin{theorem}\label{thm:main}  Given any $\Phi \in \C(M)$ the Alesker product $\nu_\mu \cdot \Phi$ is given by
\begin{equation}\label{eq:alt prod}
(\nu_\mu\cdot \Phi)(A,U) = \int_P (\mu\cdot \Phi)(A\cap \varphi_p(X), U\cap \varphi_p(X)) \, dp
\end{equation}
for any 
$ A \in \pol(M)$ and $ U\subset M $ relatively compact and Borel.
\end{theorem}
\begin{proof} As in the remark following Theorem \ref{thm:ab}, it is enough to prove the corresponding statement with the curvature measure $\Phi$ replaced by a smooth valuation $\lambda$. From Theorems \ref{thm:ab}, \ref{thm:well defined}, and \ref{thm:nu chi} it follows that
$$
\nu_\mu =\mu \cdot \nu_\chi
$$
and therefore
$$
\nu_\mu \cdot \lambda = \mu \cdot \nu_\chi \cdot \lambda =  \mu \cdot \lambda \cdot \nu_\chi =\nu_{(\mu\cdot \lambda)} 
$$
as claimed.
\end{proof}

\end{document}